\documentclass[11pt]{article}
\usepackage[a4paper]{anysize}\marginsize{3.5cm}{3.5cm}{1.3cm}{2.5cm}
\usepackage[latin1]{inputenc}

\pdfpagewidth=\paperwidth \pdfpageheight=\paperheight
\usepackage{amssymb}

\makeatletter
\renewcommand\@seccntformat[1]{\csname the#1\endcsname.\enspace}
\renewcommand\@begintheorem[2]{\trivlist\item[\hskip\labelsep{\bfseries#1 #2.}]\it}
\renewcommand\@opargbegintheorem[3]{\trivlist\item[\hskip\labelsep{\bfseries#1 #2}] {\bfseries(#3).}\enspace\it\ignorespaces}
\makeatother
\renewenvironment{abstract}{\begin{quote}\hrulefill\par\footnotesize\textbf{\abstractname.}}{\par\vskip-0.5\baselineskip\hrulefill\end{quote}}

\newtheorem{introtheorem}{Theorem}  

\newtheorem{thm}{Theorem}[section]

\newtheorem{lemma}[thm]{Lemma}
\newtheorem{proposition}[thm]{Proposition}
\newtheorem{corollary}[thm]{Corollary}

\newcommand\mkthm[2]{\newenvironment{#1}{\begin{#2}\rm}{\end{#2}}}
 \mkthm{definition}{tdefinition}
         \mkthm{remark}{tremark}
       \mkthm{example}{texample}
     \mkthm{question}{tquestion}

\newenvironment{proof}[1][Proof]{\trivlist\item[\hskip\labelsep{\textit{#1.}}]}{\hspace*{\fill}$\Box$\endtrivlist}

\newcommand\grant[1]{{\renewcommand\thefootnote{}\footnotetext{#1.}}}
\newcommand\keywords[1]{{\renewcommand\thefootnote{}\footnotetext{\textit{Keywords:} #1.}}}
\newcommand\subclass[1]{{\renewcommand\thefootnote{}\footnotetext{\textit{Mathematics Subject Classification (2010):} #1.}}}

\newcommand\C{\mathbb C}
\newcommand\Q{\mathbb Q}
\newcommand\R{\mathbb R}
\newcommand\Z{\mathbb Z}
\newcommand\N{\textnormal{N}}

\newcommand\be[1][@{\;}r@{\;}c@{\;}l@{\;}l@{\;}]{$$\everymath{\displaystyle}\renewcommand\arraystretch{1.2}\begin{array}{#1}}
\newcommand\ee{\end{array}$$}

\newcommand\compact{\itemsep=0cm \parskip=0cm}

\newcommand\matr[1]{\left(\begin{array}{*{20}{c}} #1 \end{array}\right)}

\newcommand\newop[2]{\newcommand#1{\mathop{\rm #2}\nolimits}}

\newop\End{End}
\newop\Fix{Fix}
\newop\GL{GL}
\newop\SL{SL}
\newop\fix{\#\Fix}  
\newop\tr{tr}
\newop\id{id}

\begin{document}

   \title{Fixed points and entropy of endomorphisms \\
      on simple abelian varieties}
   \author{\normalsize Thorsten Herrig}
   \date{\normalsize \today}
   \maketitle
   \thispagestyle{empty}
   \grant{The author was supported by Studienstiftung des deutschen Volkes}
   \subclass{14A10, 14K22, 14J50, 37B40}

\begin{abstract}
   In this paper we investigate fixed-point numbers and entropies of
   endomorphisms on abelian varieties. It was shown quite recently that the number of fixed-points of an iterated endomorphism on a simple complex torus is either periodic or grows exponentially. Criteria to decide whether a given endomorphism is of the one type or the other are still missing. Our first result provides such criteria for simple abelian varieties in terms of the possible types of endomorphism algebras. The number of fixed-points depends on the eigenvalues and we exactly show which analytic eigenvalues occur. This insight is also the starting point to ask for the entropy of an endomorphism. Our second result offers criteria for an endomorphism to be of zero or positive entropy. The entropy is computed as the logarithm of a real number and our third result characterizes the algebraic structure of this number.
  \end{abstract}


\section*{Introduction}

In the present paper we study the asymptotic behaviour of the number of fixed-points of iterates of a holomorphic map of an abelian variety.




The growth of the fixed-points function has been studied on two-dimensional complex tori in~\cite{BH:fp-two-dim-complex-tori}, where
a complete classification for this case can be found. In that paper it was shown that the fixed-point function grows either exponentially or periodic or a combination of both. This result was recently extended to higher dimensions by Alvarado and Auffarth \cite{AA:fp-tori}, who showed that the fixed-points function is still of one of these three types.

In this paper we extend the second theorem of~\cite{BH:fp-two-dim-complex-tori}, which deals with simple abelian surfaces, to higher dimensions. By~\cite{AA:fp-tori} we know that endomorphisms on simple abelian varieties lead to a fixed-points function which grows either periodic or exponentially. Now it is desirable to know about the exact behaviour of this function in terms of the possible types of  endomorphism algebras.

Our first result contains this information for the cases that $\End_\Q(X)$ is a totally real number field, a totally definite quaternion algebra or a CM-field. We further classify the analytic eigenvalues with respect to the fixed-point behaviour. In the case of multiplication by a totally indefinite quaternion algebra we show that, surprisingly, eigenvalues of absolute value $1$ occur that are not roots of unity -- this is in contrast with the surface case:

\begin{introtheorem}\label{thm:abelian}
 \begin{itemize}
\item[\rm(1)] Let $X$ be a simple abelian variety with $\End_{\mathbb{Q}}(X)$ isomorphic to a totally real number field, a definite quaternion algebra or a CM-field.
   Then for an endomorphism $f$ we have:
   \begin{itemize}
   \item[\rm(a)]
      Suppose that $X$ has real multiplication, i.e.,
      $\End_\Q(X)$ is a totally real number field.
      Then $\fix(f^n)$ is periodic if
      $f=\pm\id_X$, and it grows exponentially otherwise.
   \item[\rm(b)]
      Suppose that $X$ has definite quaternion multiplication,
      i.e., $\End_\Q(X)$ is of the form $F+iF+jF+ijF$ with $i^2=\alpha\in F\setminus\{0\}$, $j^2=\beta\in F\setminus\{0\}$ and $ij=-ji$, where $F$ is a totally real number field and $\End_\Q(X)\otimes_\sigma\R\simeq \mathbb{H}$ holds for every embedding $\sigma: F\hookrightarrow \R$. Write $f\in\End(X)$ as
      $f=a+bi+cj+dij$ with $a,b,c,d\in F$. Then $\fix(f^n)$ is
      periodic if
      $|a+\sqrt{b^2\alpha+c^2\beta-d^2\alpha\beta}|=1$, and
      it grows exponentially otherwise. 
   \item[\rm(c)]
      Suppose that $X$ has complex multiplication, i.e., $\End_\Q(X)$ is a CM-field.
      Then $f$ has periodic fixed-point
      behaviour if $|f|=1$, and it has
      exponential fixed-points growth otherwise.
   \end{itemize}
Moreover, if $\fix(f^n)$ is periodic, then all analytic eigenvalues are roots of unity, and if $\fix(f^n)$ grows exponentially, then they are all of absolute value $\neq 1$.
\item[\rm(2)] There exist simple abelian varieties with totally indefinite quaternion multiplication possessing endomorphisms with analytic eigenvalues of absolute value $1$ which are not roots of unity.
\end{itemize}
\end{introtheorem}

One can rephrase the first part of this theorem in the following way, using the Rosati involution: If $f\cdot f'=1$ holds, then $\fix(f^n)$ is periodic, and it grows exponentially otherwise.

Theorem \ref{thm:abelian} provides insight into the analytic eigenvalues of a given endomorphism, hence by~\cite{Gromov:entropy-holom}, \cite{Fr:entropy} and \cite{Youndin} we can compute its entropy via the formula  $$\max_{1\leq j\leq\dim X}\log\rho(f^*:\textnormal{H}^{j,j}(X)\rightarrow \textnormal{H}^{j,j}(X)),$$ where $\rho$ stands for the spectral radius.
In recent years the investigation of entropies on algebraic surfaces has attracted a lot of attention, and it is a leading question to ask whether an automorphism is of zero or positive entropy. 
Results concerning K3-surfaces (\cite{McMullen:entropy} and \cite{McMullen:dynamics}), two-dimensional complex tori \cite{Reschke:tori} and abelian surfaces \cite{Reschke:AV} show that the entropy is the logarithm of a Salem number in the case of positive entropy.

Our second result answers the question of zero or positive entropy in terms of the same types of endomorphism algebras as in Theorem~\ref{thm:abelian}. We further show that, differing from the other cases, there exist simple abelian varieties with totally indefinite quaternion multiplication which possess automorphisms whose entropy is the logarithm of a Salem number (We require that a Salem polynomial has at least one root of absolute value one (cf.~\cite{Salem}).):

\begin{introtheorem}\label{thm:entropy} 
\begin{itemize}
\item[\rm(1)] Let $X$ be a simple abelian variety with $\End_{\mathbb{Q}}(X)$ isomorphic to a totally real number field, a definite quaternion algebra or a CM-field. Then for an endomorphism $f$ we have:
\begin{itemize}
\item[\rm(a)] Suppose that $X$ has real or complex multiplication: If $f$ is of absolute value $1$, then $f$ has zero entropy, and positive entropy otherwise.
\item[\rm(b)] Suppose that $X$ has totally definite quaternion multiplication $\Big(\frac{\alpha,\beta}{F}\Big)$ and let $f$ be of the form $a+bi+cj+dij$: If $|a+\sqrt{b^2\alpha+c^2\beta-d^2\alpha\beta}|$ is $1$, then $f$ has zero entropy, and positive entropy otherwise.
\end{itemize}
\item[\rm(2)] There exist simple abelian varieties with totally indefinite quaternion multiplication possessing automorphisms of positive entropy $\log(\gamma)$, such that $\gamma$ is a Salem number.
\end{itemize}
\end{introtheorem}

 As the entropy is computed via the analytic eigenvalues, we can give an exact characterization of the numbers $\gamma$, where $\log(\gamma)$ is the entropy:

\begin{introtheorem}\label{structure} Let $X$ be a simple abelian variety with $\End_{\mathbb{Q}}(X)$ isomorphic to a totally real number field, a totally definite quaternion algebra or a CM-field, and let $f$ be an endomorphism with entropy $\log(\gamma)$. Then $\gamma$ is contained in the normal closure of the maximal totally real subfield of $\End_\Q(X)$.
\end{introtheorem}

\textit{Acknowledgements:} I would like to thank my advisor Thomas Bauer for his great support and John Voight for many valuable discussions and hints on quaternion algebras.


\section{Fixed-point formulas and endomorphism algebras}

In this section we recall fixed-point formulas, Albert's classification of the endomorphism algebra of a simple abelian variety (see~\cite[5.5.7]{BL:CAV}) and collect results about fields and quaternion algebras that will be needed in Section~\ref{sect:abelian} and~\ref{entrop}. (We include them for lack of suitable reference.)


Given an endomorphism $f$ on a $g$-dimensional simple abelian variety $X$ we have the analytic and rational representation $$\rho_a:\End(X)\to\textnormal{M}_g(\C)\quad\textnormal{and}\quad\rho_r:\End(X)\to\textnormal{M}_{2g}(\Z).$$ If we talk about analytic or rational eigenvalues, then we think of the eigenvalues of these matrices. Since we have $\rho_r\otimes 1\simeq \rho_a\oplus\overline{\rho_a}$, the Holomorphic Lefschetz Fixed-Point Formula (see~\cite{BL:fixed}) can be written as $$\fix(f^n)=\prod_{i=1}^g|1-\lambda_i^n|^2=\prod_{i=1}^g(1-\lambda_i^n)(1-\overline{\lambda_i}^n),$$ where $\lambda_1,\ldots,\lambda_g$ are the analytic eigenvalues. We immediately see that the number of fixed-points depends on the eigenvalues. To get insight into the possible eigenvalues we use the structure of the endomorphism algebra and compare the formula above with the second version of the Holomorphic Lefschetz Fixed-Point Formula (see~\cite{BL:fixed}):

Let $D=\End_\Q(X)$, $F=\textnormal{center}(D)$, $e=[F:\Q]$, $d^2=[D:F]$, and let $\textnormal{N}:D\to\Q$ be the reduced norm map. Then we have $$\fix(f)=\Big(\textnormal{N}(1-f)\Big)^{\frac{2g}{de}}.$$

Now we discuss the possible endomorphism algebras according to Albert and collect useful related results.

The first non-trivial type of $\End_\Q(X)$ is a totally real number $F$, i.e., we have $\sigma(F)\subset\R$ for all $\Q$-embeddings $\sigma: F\hookrightarrow\C$. Taking an endomorphism of such an algebra it may lie in a subfield. Then the question is whether this subfield is again totally real. The following Lemma gives an answer:

\begin{lemma}\label{lemma:subfieldtotreal} Every subfield $E$ of a totally real number field $F$ is again totally real.
\end{lemma}
\begin{proof} By the primitive element theorem we find an $\alpha$ in $F$, such that $F$ can be written as $\Q(\alpha)$. Let $\mu$ be the minimal polynomial of $\alpha$ and $K$ its splitting field which is generated by the roots of $\mu$. As all of these roots must be real by definition of $F$ and as $K$ is Galois over the rationals, the field $K$ is totally real. Further if one root of an irreducible rational polynomial lies in $K$, then all of its roots do. Finally, taking a subfield $E$ of $F$ written as $\Q(\beta)$ the primitive element $\beta$ lies in $K$ as well as all of the roots of its minimal polynomial, such that $E$ is totally real.
\end{proof}

Every algebra $\End_\Q(X)$ has a positive anti-involution, the Rosati involution with respect to a polarization. If this involution $'$ is not trivial on the center of $\End_\Q(X)$, then the pair $(\End_\Q(X),\,'\,)$ is said to be \textit{of the second kind}. Such pairs define the second non-trivial type of endomorphism algebras. We consider here only the case of CM-fields, i.e., quadratic extensions of totally real number fields which are totally imaginary.

Taking an endomorphism in an order of a CM-field $F$, our discussion about its fixed-point behaviour in Section~\ref{sect:abelian} will require knowledge about the subfields and normal closure of $F$. The following two Lemmata provide this.

\begin{lemma}\label{rem:CM} A field $F$ is totally real or a CM-field if and only if there exists an automorphism $g$ of $F$, such that $g$ coincides with the complex conjugation for every embedding $F\hookrightarrow\C$.
\end{lemma}
\begin{proof} Assume first that $F$ is totally real or a CM-field. Let $E$ be the maximal totally real subfield of $F$. For $E=F$ the assertion follows with $g=\id$. Otherwise, $g$ is the non-trivial element of Gal$(F/E)$.

Conversely, consider $E$ to be the fixed field of such an automorphism $g$. As $E$ must be totally real, the field $F$ either coincides with $E$ or is a totally imaginary quadratic extension of $E$.
\end{proof}

\begin{lemma}\label{lem:CM}
\begin{itemize}
\item[\rm(a)] The normal closure of a CM-field is again a CM-field.
\item[\rm(b)] Every subfield of a CM-field is either a totally real number field or a CM-field.
\end{itemize}
\end{lemma}
\begin{proof} 
We start with the proof of part (a): Let $\Q(\alpha_1)=F(\alpha_1)$ be a CM-field of degree $n$ with maximal totally real subfield $F$. Denoting by $\alpha_1,\ldots,\alpha_n$ all roots of the minimal polynomial of $\alpha_1$ over the rationals, every field $\Q(\alpha_i)$ has to be a CM-field, since $F$ is totally real and $\sigma(\alpha_1)$ is complex under every $\Q$-embedding $\sigma$ of $\Q(\alpha_1)$. Further, the composition of fields $$M:=\Q(\alpha_1)\cdot\ldots\cdot\Q(\alpha_n)$$ is the normal closure of $\Q(\alpha_1)$. We choose an embedding $M\hookrightarrow\C$ and write $g$ for the complex conjugation on $M$. Restricting $g$ to a field $\Q(\alpha_i)$ we get $g(\Q(\alpha_i))=\Q(\alpha_i)$ and therefore $g(\Q(\alpha_i)\Q(\alpha_j))=\Q(\alpha_i)\Q(\alpha_j)$. With respect to every embedding $$\Q(\alpha_i)\Q(\alpha_j)\hookrightarrow\C$$ the automorphism $g$ coincides with the complex conjugation on $\Q(\alpha_i)$ and $\Q(\alpha_j)$ and hence on $\Q(\alpha_i)\Q(\alpha_j)$. Now assertion (a) follows by induction and Lemma~\ref{rem:CM}.

To prove part (b) we start with a subfield $E$ of a CM-field $F$ and by the previous part we may assume that $F$ is normal over the rationals. Now we take an $h\in\textnormal{Gal}(F/E)$ and the automorphism $g$ as in Lemma~\ref{rem:CM}. For a $\Q$-embedding $\sigma:F\hookrightarrow\C$ the equation $$\overline{\cdot}\circ\sigma=\sigma\circ g$$ holds, where $\overline{\cdot}$ stands for the usual complex conjugation. Further, $\sigma\circ h$ defines another embedding of $F$, such that $$\overline{\cdot}\circ\sigma\circ h=\sigma\circ h\circ g$$ follows. Combining these two equations we get $$\sigma\circ g\circ h=\sigma\circ h\circ g$$ and by injectivity of $\sigma$ finally $g\circ h=h\circ g$. This means that $g$ commutes with every element of $\textnormal{Gal}(F/E)$ and therefore we can restrict $g$ to an automorphism on $E$. Now this restriction fulfills the condition of Lemma~\ref{rem:CM}, as every embedding of $E$ can be extended to one of $F$.
\end{proof}

The third and fourth non-trivial type of endomorphism algebras are totally definite and totally indefinite quaternion algebras:

\begin{definition} Let $B$ be a quaternion algebra over a totally real number field $F$.
\begin{itemize} 
\item[\rm(a)] $B$ is \textit{totally definite} if $$B\otimes_\sigma\R\simeq\mathbb{H}$$ holds for every embedding $\sigma:F\hookrightarrow\R$.
\item[\rm(b)] $B$ is \textit{totally indefinite} if $$B\otimes_\sigma\R\simeq\textnormal{M}_2(\R)$$ holds for every embedding $\sigma:F\hookrightarrow\R$.
\end{itemize}
\end{definition}

Moreover, every element of a quaternion algebra $B$ over a totally real number field $F$ can be written as $a+bi+cj+dij$ with $a,b,c,d\in F$, $i^2=\alpha\in F\setminus\{0\}$, $j^2=\beta\in F\setminus\{0\}$ and $ij=-ji$. Such an algebra is defined by $\alpha$ and $\beta$, and as usual we use the notation $B=\Big(\frac{\alpha,\,\beta}{F}\Big)$. The following Lemma provides alternative conditions for total definiteness and indefiniteness:

\begin{lemma}\label{crit:totdef} Let $F$ be a totally real number field and $B=\Big(\frac{\alpha, \beta}{F}\Big)$ a quaternion algebra.
\begin{itemize}
\item[\rm(a)] $B$ is totally definite if and only if $\sigma(\alpha)<0$ and $\sigma(\beta)<0$ for every $\Q$-embedding $\sigma: F\hookrightarrow\C$.
\item[\rm(b)] $B$ is totally indefinite if and only if $\sigma(\alpha)>0$ or $\sigma(\beta)>0$ for every $\Q$-embedding $\sigma: F\hookrightarrow\C$.
\end{itemize}
\end{lemma}
\begin{proof} We only prove (a), as (b) follows in the same way.

Consider an element $f=a+bi+cj+dij$ in $B$, its norm over $F$ is $a^2-b^2\alpha-c^2\beta+d^2\alpha\beta$. Assume by way of contradiction that $B$ is totally definite and that there exists a $\Q$-embedding $\sigma$ with $\sigma(\alpha)>0$. Then we find an element in $B\otimes_\sigma\R$ with Norm $-b_1^2\sigma(\alpha)<0$. Taking a sufficient large real number $a_1$ we get $a_1^2-b_1^2\sigma(\alpha)>0$ and by the intermediate value theorem there exists a real number $a_2$, such that the equation $$a_2^2-b_1^2\sigma(\alpha)=0$$ holds. Then $B\otimes_\sigma\R$ is not a skew field and cannot be isomorphic to $\mathbb{H}$, a contradiction.

Conversely, suppose $\sigma(\alpha)<0$ and $\sigma(\beta)<0$ for every $\Q$-embedding $\sigma: F\hookrightarrow\C$. Therefore it follows that $$a^2-b^2\sigma(\alpha)-c^2\sigma(\beta)+d^2\sigma(\alpha)\sigma(\beta)>0$$ for arbitrary real numbers $a,b,c$ and $d$. This means that $B\otimes_\sigma\R$ has no zero divisors and is therefore isomorphic to $\mathbb{H}$, which completes the proof.
\end{proof}

In the next Section we will construct examples of simple abelian varieties with totally indefinite quaternion multiplication which possess automorphisms whose eigenvalues are the roots of a Salem polynomial. Therefore we will need information about splitting fields of quaternion algebras. The following definition and Lemma can be found in the excellent reference~\cite{Voight:quaternion}:

\begin{definition}\cite[5.4.5]{Voight:quaternion} A quaternion algebra $B$ over $F$ is \textit{split} if $B\simeq\textnormal{M}_2(F)$. A field $K$ containing $F$ is a \textit{splitting field} for $B$ if $B\otimes_FK\simeq\textnormal{M}_2(K)$.
\end{definition}

\begin{lemma}\label{lem:indef1}\textnormal{\cite[5.4.7]{Voight:quaternion}} Let $B$ be a quaternion algebra over $F$ and let $K\supset F$ be a quadratic extension of fields. Then $K$ is a splitting field for $B$ if and only if there is an injective $F$-algebra homomorphism $K\hookrightarrow B$.
\end{lemma}



\section{Fixed points on simple abelian varieties}\label{sect:abelian}

In this Section we will develop concrete criteria which allow for a given endomorphism to decide whether its fixed-point function is periodic or grows exponentially. The first remark explains why we will consider simple abelian varieties instead of arbitrary ones.

\begin{remark}\label{simple} We show here that every algebraic integer can occur as an eigenvalue of an endomorphism of an abelian variety. In fact, given any polynomial  $$P(t)=(-1)^n(t^n+a_{n-1}t^{n-1}+\ldots+a_1t+a_0)\in\mathbb{Z}[t]$$ we can take any elliptic curve $E$. On its product $E^n$ the endomorphism $$f=\matr{0 &  & \ldots & & 0 & -a_0 \\ 1 & 0 & \ldots & & 0 & -a_1 \\  & 1 & \ddots & 0 & & \vdots \\ & & \ddots & & 0 & \vdots \\ & 0 &  & & 1 & -a_{n-1}}$$ has $P(t)$ as its characteristic polynomial.
\end{remark}

The following three propositions give rise to the first part of Theorem~\ref{thm:abelian}. We follow the classification of endomorphism algebras and investigate simple abelian varieties with real, complex and totally definite quaternion multiplication. We start with real multiplication.

\begin{proposition}\label{prop:realmult} Let $X$ be a simple $g$-dimensional abelian variety with real multiplication, i.e., $F:=\End_\Q(X)$ is a totally real number field. Then the fixed-points function of any non-zero endomorphism $f\in\End(X)$ is periodic if $f=\pm\id$, and it grows exponentially otherwise.
\end{proposition}
\begin{proof} For an endomorphism $f$ we get $\Q(1-f)$ as a subfield of $F$ which has to be totally real according to Lemma~\ref{lemma:subfieldtotreal}. Let $e:=[F:\Q]$, $m:=[F:\Q(1-f)]$, $l:=[\Q(1-f):\Q]$ and let $\sigma_1,\ldots,\sigma_l$ be the $\Q$-embeddings of $\Q(1-f)$ into the algebraic closure. Then the roots of the minimal polynomial of $1-f$ can be written as $$\sigma_1(1-f)=1-\sigma_1(f),\ldots,\sigma_l(1-f)=1-\sigma_l(f).$$ By the Holomorphic Lefschetz Fixed-Point Formula~\cite[13.1.2.]{BL:CAV} we have $$\prod_{i=1}^g(1-\lambda_i)(1-\overline{\lambda_i})=\#\Fix(f)=(\N_{F/\Q}(1-f))^{2g/e}=(\Big(\prod_{j=1}^l1-\sigma_j(f)\Big)^m)^{2g/e},$$ where $\lambda_1,\ldots,\lambda_g$ are the eigenvalues of the analytic representation $\rho_a(f)$. Considering for any integer $k$ the fixed-point number of the endomorphism $f-k+1$, the equation above becomes $$\prod_{i=1}^g(k-\lambda_i)(k-\overline{\lambda_i})=\#\Fix(f-k+1)=(\Big(\prod_{j=1}^lk-\sigma_j(f)\Big)^m)^{2g/e}.$$ Since the two polynomials coincide at every integer $k$, every eigenvalue $\lambda_i$ must coincide with an element $\sigma_j(f)$. As all $\sigma_j(f)$ are real, an eigenvalue $\lambda_i$ of absolute value $1$ must be $\pm 1$. If $1$ or $-1$ is a root of the minimal polynomial of $f$, then $f$ is $\id$ or $-\id$. Hence the function $n\mapsto\#\Fix(f^n)$ is only periodic for $f=\pm\id$ and grows exponentially otherwise.
\end{proof}

\begin{proposition}\label{prop:CM} Let $X$ be a simple $g$-dimensional abelian variety with complex multiplication, i.e., $F:=\End_\Q(X)$ is a CM-field, and let $f$ be a non-zero endomorphism. Then the fixed-point function $n\mapsto\fix(f^n)$ is periodic if $|f|=1$, and it grows exponentially otherwise.

In particular, all analytic eigenvalues are roots of unity in the periodic case and no analytic eigenvalue is of absolute value $1$ in the exponential case.
\end{proposition}
\begin{proof} If $f$ lies in the CM-field $F$, then by Lemma~\ref{lem:CM} the subfield $\Q(1-f)=\Q(f)$ is either a totally real number field or a CM-field. We may assume the second case, since the first one is treated in Proposition~\ref{prop:realmult}. Now we adopt the notation  $e:=[F:\Q]$, $m:=[F:\Q(1-f)]$ and $l:=[\Q(1-f):\Q]$ from Proposition~\ref{prop:realmult} and as in the proof of Proposition~\ref{prop:realmult} we get for any integer $k$ the equation $$\prod_{i=1}^g(k-\lambda_i)(k-\overline{\lambda_i})=\#\Fix(f-k+1)=(\Big(\prod_{j=1}^lk-\sigma_j(f)\Big)^m)^{2g/e},$$ such that every analytic eigenvalue $\lambda_i$ of $f$ has to coincide with a conjugate root $\sigma_j(f)$ of $f$. All $\lambda_i$ are in the normal closure of $F$ that is again a CM-field by Lemma~\ref{lem:CM}. In a CM-field all elements of absolute value $1$ have to be roots of unity by~\cite[Theorem 2]{Daileda:cm-field}. As all $\lambda_i$ are roots of the minimal polynomial of $f$, it suffices to test whether $f$ is of absolute value $1$ or not. 

Finally if $|f|=1$ holds, then all $\lambda_i$ are roots of unity and the function $n\mapsto\fix(f^n)$ is periodic. Otherwise, in the case $|f|\neq 1$ no $\lambda_i$ is of absolute value $1$ and the fixed-point function grows exponentially.
\end{proof}



\begin{proposition}\label{prop:totdef} Let $X$ be a simple $g$-dimensional abelian variety with totally definite quaternion multiplication $\End_\Q(X)=\Big(\frac{\alpha,\beta}{F}\Big)$ and let $f=a+bi+cj+hij$ be a non-zero endomorphism. Then the fixed-point function $n\mapsto\fix(f^n)$ is periodic if $|a+\sqrt{b^2\alpha+c^2\beta-h^2\alpha\beta}|=1$, and it grows exponentially otherwise.

In particular, all analytic eigenvalues are roots of unity in the case of periodicity and no analytic eigenvalue is of absolute value $1$ in the other case.
\end{proposition}
\begin{proof} The totally real number field $F$ is the center of $D=\Big(\frac{\alpha,\beta}{F}\Big)$ and we denote $d^2:=[D:F]$ and $e:=[F:\Q]$. According to the classification of endomorphism algebras of simple abelian varieties~\cite[5.5.7]{BL:CAV} the number $d$ has to be $2$. The endomorphism $f$ can be written as $a+bi+cj+hij$, where $i^2=\alpha,j^2=\beta$ and $ij=-ji$ with $\alpha, \beta\in F$ hold. We assume that $f$ lies in $D\setminus F$, since the case $f\in F$ is treated in Proposition~\ref{prop:realmult}.

To gain insight into the properties of the analytic eigenvalues of $f$ we use the Holomorphic Lefschetz Fixed-Point Formula~\cite[13.1.4]{BL:CAV} and get $$\fix(f)=(\textnormal{N}_{D/\Q}(1-f))^{2g/2e}=(\textnormal{N}_{F/\Q}(\textnormal{N}_{D/F}(1-f)))^{2g/2e}$$ \begin{equation}\label{equation}=(\textnormal{N}_{F/\Q}((1-a)^2-b^2\alpha-c^2\beta+h^2\alpha\beta))^{2g/2e}.\end{equation}
The reduced characteristic polynomial of $f$ is $$\textnormal{N}_{D/F}(x-f)=x^2-\textnormal{T}_{D/F}(f)x+\textnormal{N}_{D/F}(f)=(x-t_1)(x-t_2)$$ with $t_{1/2}=a\pm\sqrt{b^2\alpha+c^2\beta-h^2\alpha\beta}=a\pm\sqrt{t}$. The field $F(\sqrt{t})$ must be a quadratic extension of $F$, since $\sqrt{f}\in F$ implies $$\textnormal{N}_{D/F}(\sqrt{t}+bi+cj+hij)=t-b^2\alpha-c^2\beta+h^2\alpha\beta=0$$ which is impossible in the skew field $D$. Then taking an element $g$ of $F$ it follows that $$\textnormal{N}_{F/\Q}(g)^2=\textnormal{N}_{F(\sqrt{t})/\Q}(g).$$ By assumption, $2e$ divides $g$ (see~\cite[5.5.7]{BL:CAV}) and we can continue with (\ref{equation}) in the following way: $$\fix(f)=(\textnormal{N}_{F/\Q}((1-a)^2-b^2\alpha-c^2\beta+h^2\alpha\beta))^{2g/2e}$$ $$=(\textnormal{N}_{F/\Q}((1-t_1)(1-t_2)))^{2g/2e}=(\textnormal{N}_{F(\sqrt{t})/\Q}((1-t_1)(1-t_2)))^{g/2e}.$$ Further because of $[F(\sqrt{t}):\Q]=2e$ we have $2e$ distinct $\Q$-embeddings $\sigma_1,\ldots,\sigma_{2e}$ of $F(\sqrt{t})$ into the algebraic closure. Using the multiplicity of the norm map and the other version of the Holomorphic Lefschetz Fixed-Point Formula~\cite[13.1.2]{BL:CAV} we get $$\Big(\prod_{j=1}^{2e}(1-\sigma_j(t_1))\prod_{j=1}^{2e}(1-\sigma_j(t_2))\Big)^{g/2e}=\fix(f)=\prod_{i=1}^g(1-\lambda_i)(1-\overline{\lambda_i}),$$ where $\lambda_1,\ldots,\lambda_g$ are the analytic eigenvalues of $f$. Given an arbitrary integer $k$, this equation leads us to $$\Big(\prod_{j=1}^{2e}(k-\sigma_j(t_1))\prod_{j=1}^{2e}(k-\sigma_j(t_2))\Big)^{g/2e}=\fix(f-k+1)=\prod_{i=1}^g(k-\lambda_i)(k-\overline{\lambda_i}),$$ and hence every $\lambda_i$ has to coincide with a $\sigma_j(t_1)$ by the same argument as in the proofs of the previous Propositions.

To complete the proof we need more information about the eigenvalues $\lambda_i$ and therefore we investigate the $\Q$-embeddings of the field $F(\sqrt{t})$. We will show that this field is a CM-field, which means that its normal closure is again a CM-field by Lemma~\ref{lem:CM}. Every $\lambda_i$ lies in the normal closure of $F(\sqrt{t})$, if it is of absolute value $1$, then it has to be a root of unity (see~\cite[Theorem 2]{Daileda:cm-field}).

Taking an eigenvalue $\lambda_i$, we always find a $\Q$-embedding $\sigma_j:F(\sqrt{t})\hookrightarrow\C$, having the property that $$\sigma_j(t_1)=\sigma_j(a)+\sigma_j(\sqrt{t})=\lambda_i$$ holds. Since $a$ is an element of the totally real number field $F$, its image $\sigma_j(a)$ has again to be real. Assuming $\sigma_j(\sqrt{t})$ to be real means that $\sigma_j(\sqrt{t})^2$ is a positive real number. So we get $$\sigma_j(\sqrt{b^2\alpha+c^2\beta-h^2\alpha\beta})^2=\sigma_j(b^2)\sigma_j(\alpha)+\sigma_j(c^2)\sigma_j(\beta)-\sigma_j(h^2)\sigma_j(\alpha)\sigma_j(\beta)>0.$$ But this is impossible, because we have $\sigma_j(\alpha)<0$ and $\sigma_j(\beta)<0$ by Lemma~\ref{crit:totdef}. Therefore every $\lambda_i$ has to be complex and hence the field $F(\sqrt{t})$ and its normal closure are CM-fields which means that every $\lambda_i$ of absolute value $1$ is a root of unity. All $\lambda_i$ are roots of the minimal polynomial of $t_1$, so it suffices to test whether $t_1$ is of absolute value $1$ or not.

Finally it follows that the fixed-point function $n\mapsto\fix(f^n)$ is periodic if $|a+\sqrt{b^2\alpha+c^2\beta-h^2\alpha\beta}|=1$, and grows exponentially otherwise.
\end{proof}


In view of the previous propositions one may ask whether an endomorphism having periodic fixed-point behaviour is already an automorphisms. The following corollary answers this question:

\begin{corollary} Let $X$ be a simple abelian variety with real, totally definite quaternion or complex multiplication and let $f$ be an endomorphism with $\fix(f^n)=0$ for some $n$. Then $f$ is an automorphism.
\end{corollary}
\begin{proof} Let $f$ be an endomorphism with $\fix(f^n)=0$ for some $n$. Then in view of the fixed-point formula, $\lambda_i^n$ has to be $1$ for at least one analytic eigenvalue. If one analytic eigenvalue is a root of unity, then all eigenvalues are roots of unity. 

For $\End_\Q(X)$ a number field and $f$ an endomorphism, which is a root of unity, $f^{-1}=\overline{f}$ is also contained in $\End(X)$. If the endomorphism $f=a+bi+cj+dij$, whose analytic eigenvalues are roots of unity, lies in an order of a quaternion algebra, then $f'=a-bi-cj-dij$ also does and $$f\cdot f'=(a+\sqrt{b^2\alpha+c^2\beta-d^2\alpha\beta})(a-\sqrt{b^2\alpha+c^2\beta-d^2\alpha\beta})= 1,$$ since $|a\pm\sqrt{b^2\alpha+c^2\beta-d^2\alpha\beta}|=1$ and $a^2-b^2\alpha-c^2\beta+d^2\alpha\beta\in\R_{>0}$.
\end{proof}

\begin{remark} The converse statement is not true in general. According to Dirichlet's unit theorem (see~\cite[p. 39 - 44]{Neukirch:AZT}) there exist orders in totally real number fields and CM-fields whose groups of units contain elements of  absolute value $\neq 1$. Basic examples already exist for abelian surfaces. For instance, the ring of integers in the real quadratic number field $\Q(\sqrt{2})$ contains the unit $1+\sqrt{2}$.
\end{remark}


Concerning part two of Theorem~\ref{thm:abelian} we consider totally indefinite quaternion algebras and prove that they yield endomorphisms whose eigenvalues are the roots of a Salem polynomial. This cannot occur in the previous three cases. We provide explicit algebras and endomorphisms leading to those eigenvalues.

\begin{lemma}\label{ex:totindef} The quaternion algebra $$B:=\Big(\frac{2,\,-2-2\sqrt{13}}{\Q(\sqrt{13})}\Big)$$ is a skew field and totally indefinite. Further, there exists an order $\mathcal{O}$ in $B$ containing $f:=\frac{1}{4}(1-\sqrt{13}+i)$ whose reduced characteristic polynomial has the root $\gamma:=\frac{1}{4}(1-\sqrt{13}+\sqrt{-2-2\sqrt{13}})$, a complex conjugate of a Salem number.
\end{lemma}
\begin{proof} One can use Magma to confirm that $B$ is a skew field, using the following code:

\begin{verbatim}
_<x> := PolynomialRing(Rationals());
K<sqrt13> := NumberField(x^2-13);
B := QuaternionAlgebra(K, 2, -2-2*sqrt13);
IsMatrixRing(B);
\end{verbatim}

By definition $i^2=-2-2\sqrt{13}$ and $j^2=2$ hold. Identifying $i$ with $\sqrt{-2-2\sqrt{13}}$ and $j$ with $\sqrt{2}$, there exist injective $\Q(\sqrt{13})$-algebra homomorphisms $$\iota_1:L:=\Q(\sqrt{13})(\sqrt{2})\hookrightarrow B \quad \textnormal{and} \quad \iota_2:K:=\Q\Big(\sqrt{-2-2\sqrt{13}}\Big)\hookrightarrow B.$$ So the fields $K$ and $L$ define splitting fields for $B$ by Lemma~\ref{lem:indef1}. The image of $2$ is positive under every $\Q$-embedding of $\Q(\sqrt{13})$, hence $B$ is totally indefinite by Lemma~\ref{crit:totdef}.

If we denote by $R$ the ring of integers in $\Q(\sqrt{13})$ and by $S$ the ring of integers in $K$, then $S$ constitutes an $R$-order in $K$. Now $\mathcal{O}:=S\oplus Sj$ becomes an $R$-order in $B$ via $\iota_2$. The number $$\gamma:=\frac{1}{4}(1-\sqrt{13}+\sqrt{-2-2\sqrt{13}}),$$ a complex root of the Salem polynomial $$x^4-x^3-x^2-x+1,$$ is contained in $S$ and corresponds to $$f:=\frac{1}{4}(1-\sqrt{13}+i)$$ via $\iota_2$. Considering the reduced characteristic polynomial of $f$ we get $$\textnormal{N}_{B/\Q(\sqrt{13})}(x-f)=x^2-\textnormal{T}_{B/\Q(\sqrt{13})}(f)x+\textnormal{N}_{B/\Q(\sqrt{13})}(f)$$ $$=x^2-\frac{1}{4}(1-\sqrt{13})x+1=(x-t_1)(x-t_2)$$ with $t_1=\gamma$ and $t_2=\overline{\gamma}$.
\end{proof}

\begin{proposition}\label{prop:totindef} There exists a simple abelian variety with totally indefinite quaternion multiplication whose endomorphism ring contains an element with an eigenvalue of absolute value one which is not a root of unity.
\end{proposition}
\begin{proof} First we denote $F:=\Q(\sqrt{13})$ and adopt the notation of Lemma~\ref{ex:totindef}. According to~\cite[p. 254 - 257]{BL:CAV} there is an abelian variety $X$ of dimension $4$ with $\mathcal{O}\subseteq \End(X)$. By construction its endomorphism algebra is $B$, and as $B$ is a skew field, $X$ has to be simple.

Considering the endomorphism $f\in\mathcal{O}$, the roots of its reduced characteristic polynomial are $t_1=\gamma$ and $t_2=\overline{\gamma}$, which has been shown in Lemma~\ref{ex:totindef}. As $K$ is a quadratic extension of $F$ and contains $\gamma$ and $\overline{\gamma}$, we get $$(\textnormal{N}_{F/\Q}(1-t_1)(1-t_2))^2=\textnormal{N}_{K/\Q}(1-t_1)(1-t_2).$$ Computing the norm and using both versions of the Fixed-Point Formula we have $$\Big(\prod_{j=1}^41-\sigma_j(t_1)\Big)\Big(\prod_{j=1}^41-\sigma_j(t_2)\Big)=\fix(f)=\prod_{i=1}^4(1-\lambda_i)(1-\overline{\lambda_i}),$$ where $\sigma_1,\ldots,\sigma_4$ are the $\Q$-embeddings of $K$ into $\C$. As in the proof of Proposition~\ref{prop:realmult} we see that $\gamma=t_1$ coincides with an eigenvalue $\lambda_i$. This element is of absolute value one, but not a root of unity.
\end{proof}

\begin{remark} We show here that there are other totally indefinite quaternion algebras with the same properties as the one in Lemma~\ref{ex:totindef} which can be verified in an analogous way. Moreover, we can copy the proof of Proposition~\ref{prop:totindef} and hence get more simple abelian varieties that provide endomorphisms with eigenvalues of absolute value one which are not roots of unity.
\begin{itemize}
\item[\rm(a)]
The totally indefinite quaternion algebra $$\Big(\frac{2,\,94-14\sqrt{61}}{\Q(\sqrt{61})}\Big)$$ has an order with the element $\frac{1}{4}(7-\sqrt{61}+i)$ for $i^2=94-14\sqrt{61}$ which provides the analytic eigenvalue $$\frac{1}{4}(7-\sqrt{61}+\sqrt{94-14\sqrt{61}}),$$ a complex root of the Salem polynomial $x^4-7x-x^2-7x+1$.
\item[\rm(b)]
The totally indefinite quaternion algebra $$\Big(\frac{2,\,10-6\sqrt{17}}{\Q(\sqrt{17})}\Big)$$ defined by $i^2=10-6\sqrt{17}$ contains an order with the element $\frac{1}{4}(3-\sqrt{17}+i)$ which leads to the eigenvalue $$\frac{1}{4}(3-\sqrt{17}+\sqrt{10-6\sqrt{17}}),$$ a complex root of the Salem polynomial $x^4-3x^3-3x+1$.
\end{itemize}
\end{remark}

The next remark shows that the conditions from Proposition~\ref{prop:totdef} do not hold in the totally indefinite case.

\begin{remark}\label{bsp:indef} The endomorphisms $$\frac{1}{4}(1-\sqrt{13}+i),\quad \frac{1}{4}(7-\sqrt{61}+i)$$ and $$\frac{1}{4}(3-\sqrt{17}+i)$$ meet the conditions of Proposition~\ref{prop:totdef}, but they have an exponentially growing fixed-point function, because one eigenvalue is a Salem number, an element of absolute value bigger than $1$.
\end{remark}

\begin{proof}[Proof of Theorem~\ref{thm:abelian}] The first part of Theorem~\ref{thm:abelian} follows by Proposition~\ref{prop:realmult}, Proposition~\ref{prop:CM} and Proposition~\ref{prop:totdef}. The second part is shown in Proposition~\ref{prop:totindef}.
\end{proof}

\section{Entropy of endomorphisms on simple abelian varieties}\label{entrop}

The entropy of a surjective holomorphic endomorphism $f$ on a compact K\"ahler manifold can be computed as the natural logarithm of the spectral radius of the action on the cohomology induced by $f$ (see \cite{Fr:entropy}, \cite{Gromov:entropy-holom} and \cite{Youndin}).
Note that if $X$ is a simple abelian variety, then every non-zero endomorphism is surjective.

Now we characterize the algebraic structure of $\gamma$, where $\log(\gamma)$ is the entropy of an endomorphism. The next proposition proves Theorem~\ref{structure}:

\begin{proposition}\label{prop:ent} Let $X$ be a simple $n$-dimensional abelian variety with real, totally definite quaternion or complex multiplication and let $f$ be an endomorphism of entropy $\log(\gamma)$. Let $K$ be the field generated by the analytic eigenvalues of $f$. Then $\gamma$ is contained in the maximal totally real subfield $L$ of $K$.

More precisely, the normal closure of the maximal totally real subfield of $\End_\Q(X)$ contains $L$.
\end{proposition}
\begin{proof} The entropy of an endomorphism $f$ on $X$ is $$\max_{1\leq j\leq n}\log\rho(f^*:\textnormal{H}^{j,j}(X)\rightarrow \textnormal{H}^{j,j}(X))=\log(\gamma),$$ where $\rho$ stands for the spectral radius. By the proofs of Proposition~\ref{prop:totdef}, Proposition~\ref{prop:CM} and Proposition~\ref{prop:realmult} we know that either all eigenvalues $\lambda_i$ are real or complex. The eigenvalues of $f^*$ on $\textnormal{H}^{j,j}(X)$ equal a product of $2j$ pairwise distinct rational eigenvalues. ( Let $\mu_1,\ldots,\mu_{2n}$ be the rational eigenvalues. By distinct we mean that none of these eigenvalues occurs twice in a product, but $\mu_j=\mu_i$ is still possible for $j\neq i$.)

For $X$ having real multiplication our assertion follows directly, as all eigenvalues are real as roots of the minimal polynomial of $f$.

Consider $X$ to have totally definite quaternion or complex multiplication. Let $\lambda_1,\ldots,\lambda_i$ be all analytic eigenvalues of absolute value bigger than one. Then we get $$\gamma=|\lambda_1|^2\cdot\ldots\cdot|\lambda_i|^2.$$
As all $\lambda_j$ lie in a CM-field, which is closed under complex multiplication, every product $\lambda_j\cdot\overline{\lambda_j}=|\lambda_j|^2$ is contained in the maximal totally real subfield. If $X$ has complex multiplication, then we are done. If $X$ has totally definite quaternion multiplication by $\Big(\frac{\alpha,\beta}{F}\Big)$, then we know by the proof of Proposition~\ref{prop:totdef} that all $\lambda_j$ lie in the normal closure of the CM-field $F(\sqrt{t})$ which is again a CM-field. Hence all $|\lambda_j|^2$ are in the normal closure of $F$.
\end{proof}

\begin{remark} If $X$ is a simple abelian variety with totally definite quaternion multiplication in $B=\Big(\frac{\alpha,\beta}{F}\Big)$, then $F$ is the maximal totally real subfield. Indeed, if there exists a real quadratic extension $K$ of $F$ contained in $B$, then $K$  defines a splitting field for $B$ by Lemma~\ref{lem:indef1}. But the existence of such a field is a contradiction to the totally definiteness of $B$.
\end{remark}

As a number field $\Q(\alpha)$ defined by a Salem number $\alpha$ cannot be totally real, we immediately get:

\begin{corollary} Let $X$ be a simple $n$-dimensional abelian variety with real, totally definite quaternion or complex multiplication and let $f$ be an endomorphism of entropy $\log(\gamma)$. Then $\gamma$ is not a Salem number.
\end{corollary}

Using our knowledge about the analytic eigenvalues we are now able to determine which endomorphisms are of positive entropy and which of zero.

\begin{corollary}\label{cor1:entropy}
Let $X$ be a simple abelian variety.
\begin{itemize}
\item[\rm(a)] Let $X$ have real or complex multiplication and let $f$ be an endomorphism of $X$. Then:

If $f$ is of absolute value $1$, then $f$ has zero entropy, and $f$ has positive entropy otherwise.
\item[\rm(b)] Let $X$ have totally definite quaternion multiplication $\Big(\frac{\alpha,\beta}{F}\Big)$ and let $f=a+bi+cj+dij$ be an endomorphism of $X$. Then:

If $|a+\sqrt{b^2\alpha+c^2\beta-d^2\alpha\beta}|$ is $1$, then $f$ has zero entropy, and $f$ has positive entropy otherwise.
\end{itemize}
\end{corollary}
\begin{proof} In view of Proposition~\ref{prop:realmult}, Proposition~\ref{prop:CM} and Proposition~\ref{prop:totdef} and considering $|f|=1$ or $|a+\sqrt{b^2\alpha+c^2\beta-d^2\alpha\beta}|=1$, every analytic eigenvalue has to be of absolute value $1$ and zero entropy follows. 

If we consider $|f|\neq 1$ or $|a+\sqrt{b^2\alpha+c^2\beta-d^2\alpha\beta}|\neq 1$, then the same propositions tell us that at least one analytic eigenvalue is of absolute value bigger than $1$ which gives rise to positive entropy.
\end{proof}

Now the question is whether there exist automorphisms of positive entropy:

\begin{corollary} There exist simple abelian varieties with real, totally definite quaternion and complex multiplication which have automorphisms of positive entropy.
\end{corollary}
\begin{proof} Let $F$ be a totally real number field of degree $r$ or a CM-field of degree $2r$. Then according to Dirichlet's unit theorem (see~\cite[p. 39 - 44]{Neukirch:AZT}) every maximal order $\mathcal{O}$ in $F$ has exactly $r-1$ fundamental units $\epsilon_1,\ldots,\epsilon_{r-1}$, such that any unit $\epsilon$ can be written uniquely as $$\epsilon=\zeta\cdot\epsilon_1^{m_1}\cdot\ldots\cdot\epsilon_{r-1}^{m_{r-1}}$$ with a root of unity $\zeta$ and integers $m_i$. For $r>1$ there are obviously units of absolute value bigger than $1$ which yield automorphisms of positive entropy, e.g., $1+\sqrt{2}$ in the real quadratic number field $\Q(\sqrt{2})$.
\end{proof}

In view of Proposition~\ref{prop:totindef} we know that there exist endomorphisms of simple abelian varieties whose analytic eigenvalues are roots of a Salem polynomial. The following corollary shows that such endomorphisms are automorphisms and lead to a new type of positive entropy.

\begin{corollary}\label{end} There exist simple abelian varieties with totally indefinite quaternion multiplication which provide automorphisms with positive entropy $\log(\gamma)$, where $\gamma$ is a Salem number.
\end{corollary}
\begin{proof}The endomorphisms in Proposition~\ref{prop:totindef} or Remark~\ref{bsp:indef} are automorphism, because they are inverse to the quaternion conjugation of themselves. Their eigenvalues are always roots of a Salem polynomial. Via the formula in the proof of Proposition~\ref{prop:ent} the entropy is given by the logarithm of the product of the rational eigenvalues of absolute value bigger than $1$. In our situation the only analytic eigenvalue of absolute value bigger than $1$ is a Salem number, hence the entropy is the logarithm of its square. As the square of a Salem number is again a Salem number, we are done.
\end{proof}

\begin{proof}[Proof of Theorem~\ref{thm:entropy}] The first statement is proved by Corollary~\ref{cor1:entropy} and the second by Corollary~\ref{end}.
\end{proof}





\footnotesize
   
   \bigskip
   Thorsten Herrig,
   Fachbereich Mathematik und Informatik,
   Philipps-Universit\"at Marburg,
   Hans-Meerwein-Stra\ss e,
   D-35032 Marburg, Germany.

   \nopagebreak
   \textit{E-mail address:} \texttt{herrig106@mathematik.uni-marburg.de}


\end{document}